\documentclass{article}

\usepackage{amsmath}
\usepackage{amsfonts}
\usepackage{amssymb}
\usepackage{booktabs}

\newlength{\spc} 

\newcommand{\states}{\mathcal{S}}
\newcommand{\actions}{\mathcal{A}}

\usepackage[mathscr]{euscript}

\newcommand{\tss}[1]{\textrm{#1}}

\usepackage{PRIMEarxiv}
\usepackage{amsmath}
\usepackage[utf8]{inputenc} 
\usepackage[T1]{fontenc}    
\usepackage[colorlinks=true, linkcolor=black]{hyperref}       
\usepackage{url}            
\usepackage{booktabs}       
\usepackage{amsfonts}       
\usepackage{nicefrac}       
\usepackage{microtype}      
\usepackage{lipsum}
\usepackage{fancyhdr}       
\usepackage{graphicx}       
\graphicspath{{media/}}     
\usepackage{amsthm}
\theoremstyle{plain}
\newtheorem{theorem}{Theorem}[section]
\newtheorem{lemma}[theorem]{Lemma}
\newtheorem{corollary}[theorem]{Corollary}
\title{Welfare Maximization Algorithm for Solving Budget-Constrained Multi-Component POMDPs
}

\author{
  Manav Vora\\
  Department of Aerospace Engineering \\
  University of Illinois Urbana-Champaign \\
  \texttt{mkvora2@illinois.edu} \\
   \And
  Pranay Thangeda \\
  Department of Aerospace Engineering \\
  University of Illinois Urbana-Champaign \\
  \texttt{pranayt2@illinois.edu} \\
    \And
  Michael N. Grussing \\
   Engineer Research and Development Center \\
  US Army Corps of Engineers \\
  \texttt{michael.n.grussing@erdc.dren.mil} \\
     \And
  Melkior Ornik \\
  Department of Aerospace Engineering \\
  University of Illinois Urbana-Champaign \\
  \texttt{mornik@illinois.edu} \\
}

\begin{document}
\maketitle

\begin{abstract}
Partially Observable Markov Decision Processes (POMDPs) provide an efficient way to model real-world sequential decision making processes. Motivated by the problem of maintenance and inspection of a group of infrastructure components with independent dynamics, this paper presents an algorithm to find the optimal policy for a multi-component budget-constrained POMDP. We first introduce a budgeted-POMDP model (\textit{b-POMDP}) which enables us to find the optimal policy for a POMDP while adhering to budget constraints. Next, we prove that the value function or maximal collected reward for a special class of b-POMDPs is a concave function of the budget for the finite horizon case. Our second contribution is an algorithm to calculate the optimal policy for a multi-component budget-constrained POMDP by finding the optimal budget split among the individual component POMDPs. The optimal budget split is posed as a welfare maximization problem and the solution is computed by exploiting the concavity of the value function. We illustrate the effectiveness of the proposed algorithm by proposing a maintenance and inspection policy for a group of real-world infrastructure components with different deterioration dynamics, inspection and maintenance costs. We show that the proposed algorithm vastly outperforms the policies currently used in practice.
\end{abstract}


\section{Introduction}
Sequential decision-making is an integral component of many real world problems like machine maintenance, structural inspection and autonomous robotics \cite{Cassandra2003ASO}. Markov Decision Processes (MDPs) have provided an efficient framework to model and solve such problems while accounting for the corresponding uncertainty \cite{putterman}. POMDPs are a generalized version of MDPs, allowing for more uncertainty to be accounted for in the form of partial observability of the state of the system \cite{astrompomdp}. However, finding optimal policies for POMDPs is much more computationally intensive as compared to MDPs and has been proven to be PSPACE-complete \cite{complexitypomdp}. Synthesis of optimal policies for POMDPs is a classical problem and many algorithms have been proposed for the same \cite{finitestatecontroller, pruning, heuristicVI, NIPS2010_edfbe1af}.

This paper considers optimal planning for a class of structured budget-constrained POMDPs. This setting is motivated by infrastructure maintenance planning -- a widely studied problem \cite{mishra,infra,hcrl} that involves finding the optimal policy for maintenance and inspection, of an infrastructure component or a group of components within a certain budget\cite{deeprl,subway}. For simplicity of planning, the stochastic dynamics of multi-component systems are modelled using a POMDP \cite{intinspect,maintplan}. Also, the dynamics of individual components are assumed to be independent of each other \cite{ANDRIOTIS2021107551}. We will thus model such a setting by multi-component budget-constrained POMDPs where the transition probabilities of the individual component POMDPs are decoupled from each other. Thus we can say that the individual component POMDPs are weakly-coupled in the sense that they are only connected by the shared total budget of the multi-component POMDP. 


An algorithm for solving cost-constrained POMDPs has been proposed in \cite{CC-POMCP}. However, this algorithm becomes computationally infeasible for multi-component POMDPs with very large state spaces. A POMDP-based solution for optimal maintenance and inspection of structures using Dynamic Bayesian Networks is presented in \cite{optmaint}. However, in addition to computational infeasibility, this algorithm does not account for budget constraints. Optimal allocation for MDPs has been studied in \cite{stochasticranking}. The paper models the statistical ranking and selection problem as an MDP and derives an approximately optimal allocation policy using value function approximation. In our work, we study optimal budget allocation for multi-component POMDPs, for infrastructure management. A method for solving budget-constrained MDPs has been presented in \cite{cmdps}. The paper introduces a budgeted-MDP model which includes the budget as an implicit part of the state. This is because the paper uses a cost function, similar to Constrained MDPs \cite{cmdp1}, to keep a track of the cost incurred by the policy. Hence, this algorithm cannot be directly extended to POMDPs because the partial observability of the state would cause a violation of the budget constraint in some cases. 

In this work, we propose a computationally efficient algorithm for optimal policy synthesis of a multi-component budget-constrained POMDP.
Our contributions here are:
\begin{itemize}
    \item we introduce a b-POMDP model to facilitate strict adherence to budget constraints in POMDPs,
    \item we obtain an approximately optimal policy for multi-component POMDPs by finding the optimal budget split among the individual component POMDPs.
\end{itemize}
The b-POMDP model includes the total cost incurred upto a time instant $k$ explicitly as a part of the state vector. We show that the value function for a particular class of b-POMDPs is a concave function of the budget. Next, we compute the optimal policy for the individual component POMDPs by modeling them as b-POMDPs and using an online solver like POMCP \cite{NIPS2010_edfbe1af}. Doing so gives us the approximate maximal total reward collected by the policy in terms of the budget allocated to the b-POMDP. We use these rewards to calculate the optimum distribution of the total budget among the individual component b-POMDPs and find the approximately optimal policy of the multi-component POMDP. The budget splitting is posed as a welfare maximization problem constrained by the total budget of the multi-component POMDP. The concave nature of the value function renders this as a convex optimization problem, thus guaranteeing a global optimum. We demonstrate the utility of the proposed algorithm by finding optimal maintenance and inspection policies for multiple components of a realistic general administrative building, subject to a budget. Based on this real data, we show that our algorithm vastly outperforms the policy currently used in practice.

\section{Preliminaries and Background}
\label{sec:prelim}
In this section, we provide background on Partially Observable Markov Decision Processes for sequential decision-making with stochastic dynamics. We start by defining the notation used in the paper. 

Given a finite set $\actions$, $|\actions|$ denotes its cardinality and $\Delta(\actions)$ denotes the set of all probability distributions over the set $\actions$. Notation $\mathbb{N}_0$ denotes the set of natural numbers including 0 i.e. $\mathbb{N}_0 = \{0,1,2, \ldots\}$. The symbols $\lfloor . \rfloor$ and $\lceil . \rceil$ denote the floor  and ceiling functions respectively.

\subsection{Partially Observable Markov Decision Process}
A discrete-time finite-horizon POMDP \cite{pomdpcassandra, Braziunas-POMDPsurvey} $M$ is specified by the 8-tuple $(\states, A, \pi, T, \Omega, O, R, H)$, where $\states$ denotes a finite set of states, $A$ denotes a finite set of actions, $\pi : \states \rightarrow A$ denotes the policy which specifies the action to take in a given state $s$ and $T :  \states\times A \rightarrow \Delta(\states)$ denotes the transition probability function, where $\Delta(\states)$ is the space of probability distributions over $\states$. Furthermore, $\Omega$ denotes a finite set of observations and $O : \Omega\times\states\times A \rightarrow \Delta(\Omega)$ denotes the observation probability function where $\Delta(\Omega)$ is analogous to $\Delta(\states)$. Finally, $R : \states \times A \rightarrow [0, R_\tss{max}]$ denotes the reward function and $H \in \mathbb{N}_0$ denotes the finite planning horizon.

For the above POMDP, at each time step, the environment is in some state $s \in \states$ and the agent interacts with the environment by taking an action $a \in A$. Doing so results in the environment transitioning to a new state $\bar{s} \in \states$ in the next time step with probability $T(s,a,\bar{s})$. Simultaneously, the agent receives an observation $o \in \Omega$ regarding the state of the environment with probability $O(o|\bar{s},a)$ which depends on the new state of the environment and the action taken by the agent. In a POMDP the agent doesn't have access to the true state of the environment. However, the agent can update it's belief about the true state of the environment using this observation. The agent also receives a reward $R(s,a)$. 

The problem of optimal policy synthesis for a finite-horizon POMDP is that of choosing a sequence of actions which maximizes the expected total reward. 

\section{Problem Formulation}
\label{sec:problem}
We consider a multi-component POMDP, which is a collection of $n$ component POMDPs. The component POMDPs are weakly-coupled in the sense that they have independent transition probabilities and are connected only by the shared total budget. In this paper, we consider optimal policy synthesis for POMDPs with budget, i.e., each action incurs a cost and the total cost incurred by the optimal policy is limited by the budget. We first formally define the multi-component POMDP with a budget and then define the problem of finding the optimal policy for such a POMDP.

\subsection{Multi-Component Decoupled POMDP with Shared Budget}
\label{subsec:multicomponent}
For a multi-component POMDP, the state space $\states \subseteq \mathbb{N}_0^n$ is given by $\states = \states^1 \times \states^2 \times \ldots \times \states^n$. The state space $\states^i \in \mathbb{N}_0$, for component $i$, is given by $\states^i = \{0,1,2, \ldots, s_{max}\}$, where $s_{max} \in \mathbb{N}_0$.The state $s_k \in \states$, at time step $k$ is given by $s_k = \{s^1_k,s^2_k, \ldots, s^n_k\}$ where $s^i_k \in \states^i$ represents the state of component $i$ at time step $k$. 
The action space is given by 
\begin{equation*}
    A = \prod_{i = 1}^n A^i,
\end{equation*}
where the action space $A^i$ for component $i$ is given by $A^i = \{d^i,q^i,m^i\}$. Action $d^i$ lets the component move to a new state according to the transition probabilities. The action $q^i$ provides an observation which is equal to the next state $\bar{s}^i$ of the component and action $m^i$ drives the component state to $s_{max}$. The transition probability function for the multi-component POMDP for $s,\bar{s} \in \states$ and $a \in A$ is given by
\begin{equation*}
    T(s,a,\bar{s}) = \prod_{i=1}^n T^i(s^i,a^i,\bar{s}^i).
\end{equation*}
$T^i$ denotes the transition probability function for component $i$ and is defined as
\begin{equation}
    T^i(s,a,\bar{s}^i) = \begin{cases}
    1,& \text{if } \bar{s}^i = s_{max} \text{ and } a^i = m^i, \\
    p^i(s^i,a^i,\bar{s}^i),& \text{if } \bar{s}^i \leq s^i \text{ and } a^i \in \{d^i, q^i\},\\
    1,& \text{if } \bar{s}^i = 0 = s^i \text{ and } a^i \in A^i,\\
    0,& \text{otherwise}.
    \end{cases}
\end{equation}
The probability $p^i(s^i,a^i,\bar{s}^i)$ is chosen according to a probability distribution specific to component $i$. From the above equation, it can be observed that $0$ is an absorbing state. 

The observation space is given by $\Omega = \states \cup \{e\}$, where $e \in \mathbb{N}_0$ is an observation that does not provide any information regarding the true state of the system, i.e., $e \notin \states^i$ for all $i \in \{1,2,\ldots,n\}$. The observation function for the multi-component POMDP is given by 
\begin{equation*}
    O(\bar{s},a,o) = \prod_{i=0}^n O^i(\bar{s}^i,a^i,o^i).
\end{equation*}
Here, $O^i$ is the observation probability function for component $i$ and is defined as
\begin{equation*}
    O^i(\bar{s}^i,a^i,o^i) = \begin{cases}
        1,& \text{if } o^i = \bar{s}^i \text{ and } a^i \in \{q^i, m^i\} \\
        1,& \text{if } o^i = e \text{ and } a^i = d^i \\
        0,& \text{otherwise}.
    \end{cases}
\end{equation*}

For each component $i$, each action $d^i,q^i$ and $m^i$ incurs a cost $c_d^i, c_q^i$ and $c_m^i$ respectively, against a total budget $B$. 

\subsection{Problem Statement}
For a multi-component POMDP, given by the formulation in the previous section, we consider the problem of finding an optimal policy $\pi^*$ which maximizes the time before reaching the absorbing state. Mathematically, $\pi^*$ maximizes $t$ such that $s_t > 0$. Furthermore, for a horizon of length $H$, $\pi^*$ should be such that the total cost incurred for the multi-component POMDP does not exceed the total budget. 
We propose an approximately optimal solution for the above problem through a two-step approach. In the first step, we will solve a single-component POMDP for any given budget. In the second step, we will partition the total budget $B$ into budgets for each individual component.

\section{Solution Approach}
\label{sec:solution}
In this section, we detail our methodology for solving the problem of optimal policy synthesis of a multi-component POMDP. First, we introduce the b-POMDP model and discuss how to solve a single b-POMDP. Next, we discuss why the value function for such a POMDP is a concave function of the budget. Finally, we present our proposed method for finding the optimal policy for an $n$-component POMDP by computing the optimal split of the total budget, among the individual component POMDPs.

\subsection{Budgeted-POMDP Model (b-POMDP)}
\label{subsec:budgeted-POMDP}

Our main goal is to find an optimal policy for a POMDP while adhering to a total budget for actions. The budgeted-MDP model in \cite{cmdps} tracks the incurred cost using a cost function similar to Constrained MDPs \cite{cmdp1}. This model can't be extended directly to a POMDP because the partial observability of the state would lead to budget violation in some cases. Hence, we introduce a new b-POMDP model. In a b-POMDP, the budget constraint is incorporated by augmenting the total cost incurred upto time instant k, to each state of the state space. Thus, if we consider a single component POMDP with a total budget $B$, the modified state at time instant $k$ according to the b-POMDP formulation is given by $(s_k,c_k)$ where $s_k$ is as defined in the previous section, for $i = 1$. We assume that unlike $s_k$, $c_k$ is completely observable at all time instants. The transition function for the cost component of the state is given by:
\begin{equation*}
    T_c(c^{\prime}|c,a) = \begin{cases}
    1,& \text{if } c^{\prime} = c+c_m \text{ and } a = m \\
    1,& \text{if } c^{\prime} = c+c_q \text{ and } a = q \\
    1,& \text{if } c^{\prime} = c+c_d \text{ and } a = d \\
    0,& \text{ otherwise }.
    \end{cases}
\end{equation*}
The transition function for the overall b-POMDP is:
\begin{equation}
    T^\prime((s,c),a,(\bar{s},c^\prime)) = T(s,a, \bar{s})T_c(c,a,c^{\prime}),
\end{equation}
where $T(s,a,\bar{s})$ is defined in Section \ref{subsec:multicomponent}.
The new formulation prevents the policy from violating the budget at any time instant $k$. This is done 
by making $c_k > min\{B-c_m,B-c_i\}$ an absorbing state, similar to $s=0$. The reward function for the b-POMDP is given by:
\begin{equation*}
    R^\prime((s,c),a) = \begin{cases}
        r_1 > 0,& \text{if } s > 0, \\
        r_2 = 0,& \text{if } s = 0.
        \end{cases}
\end{equation*}

To find an optimal policy for a b-POMDP, we use the method of Monte-Carlo Planning in POMDPs (POMCP \cite{NIPS2010_edfbe1af}). POMCP is an online planning algorithm for large POMDPs, which combines a Monte-Carlo update of the agent’s belief with a Monte-Carlo tree search for the best action from the current belief state. 
For a b-POMDP, the maximal collected reward (value function) obtained using POMCP will be a function of the budget $B$ associated with it. We will now prove that this value function is concave in the budget for a special subclass of our overall problem.

\subsection{Proof for Concavity of Optimal Value Function}
\label{subsec:proof}
Consider an MDP with state space $\states_{MDP} = \{0,1,2 \ldots s_{max}\}$, where $s_{max} \in \mathbb{N}_0$ and action space is $A_{MDP} = \{m,d\}$. The state of the system decreases by a value $d_0 \in \mathbb{N}_0$ unless we perform action $m$. The transition probability function is defined as:
\begin{equation*}
    T(s^{\prime}|s,a) = \begin{cases}
    1,& \text{if } s^{\prime} = s_{max} \text{ and } a = m \\
    1,& \text{if } s^{\prime} = s-d_0 \text{ and } a = d\\
    1,& \text{if } s^{\prime} = 0 = s \text{ and } a \in A_{MDP}\\
    0,& \text{otherwise,}
    \end{cases}
\end{equation*}
From the above transition function, we can clearly see that state $0$ is an absorbing state. 
Also, $d_0$ is a decrease in the state value. The cost for the $d$ action is $c_d = 0$ and the cost for the $m$ action is $c_m > 0$. For simplicity, assume that $c_m = 1$. Let the available budget be denoted by $b \in \mathbb{N}_0$. This means that we can perform the action $m$ at most $b$ times. 

The reward function is similar to our original problem with a constant positive reward $r$ for all states $s > 0$ and a zero reward for $s=0$.

The value function for an MDP is the total expected reward collected by the optimal policy. For a b-POMDP, the value function is a function of the state value and the available budget. Let $V_H(s,b)$ represent the value function of the state and budget for a given horizon with $H$ steps to go.

The following lemma will be used for proving the concavity of the value function for the above mentioned MDP with respect to the budget.

\begin{lemma} \label{lemma:1}
For a given budget $b$ and horizon $H$, the value function $V_H(s_0,b)$ is an increasing function of the state $s_0$, i.e., for two states $s_0$ and $s_0^\prime$ such that $s_0 < s_0^\prime$, the following holds:
\begin{equation*}
    V_H(s_0^\prime,b) \geq V_H(s_0,b).
\end{equation*}
\end{lemma}
\begin{proof}
Let $\pi^*$ be a policy that generates the value function $V_H(s_0,b)$, in the sense of maximizing the expected total return over the horizon $H$ given the initial state $s_0$ and budget $b$. Thus, we have $V_H^{\pi^*}(s_0,b) = V_H(s_0,b)$. If we start at state $s_0^\prime > s_0$, then following the same policy $\pi^{*}$ the expected return will be at least as good as that for $s_0$. We can therefore say that:
\begin{equation*}
    E\left[\sum_{t=0}^{H} r_t \Big\vert s_0^\prime, \pi^{*}, b\right] \geq E\left[\sum_{t=0}^{H} r_t \Big\vert s_0, \pi^*, b\right].
\end{equation*}
Hence, we can say that for a given budget $b$ and horizon $H$, the optimal policy for $s_0^\prime$ is atleast as good as the optimal policy for $s_0$ and hence:
\begin{equation*}
    V_H(s_0^\prime,b) \geq V_H^{\pi^*}(s_0^\prime,b) \geq V_H^{\pi^*}(s_0,b) = V_H(s_0,b) \quad \forall s_0^\prime > s_0
\end{equation*}
\end{proof}
\begin{lemma}\label{lemma:2}
    For a given initial state $s_0$ and horizon $H$, the value function $V_H(s_0,b)$ is an increasing function of the budget $b$, i.e., for two budgets $b$ and $b^\prime$ such that $b < b^\prime$, the following holds:
\begin{equation*}
    V_H(s_0,b^\prime) \geq V_H(s_0,b).
\end{equation*}
\end{lemma}
\begin{proof}
    Let $\pi^*$ be a policy that generates the value function $V_H(s_0,b)$, in the sense of maximizing the expected total return over the horizon $H$ given the initial state $s_0$ and budget $b$. The same policy will also be optimal for the same initial state $s_0$ but budget $b^\prime$. This is because following the same policy will result in the same trajectory of states for both budgets, but with a higher budget at each time step for $b^\prime$. Thus, we have
    \begin{equation*}
         E\left[\sum_{t=0}^{H} r_t \Big\vert s_0, \pi^{*}, b^\prime\right] \geq E\left[\sum_{t=0}^{H} r_t \Big\vert s_0, \pi^*, b\right].
    \end{equation*}
    Hence, the value function generated by this policy with budget $b^\prime$ will be at least as good as the value function generated by the same policy with budget $b$. Since this holds for any optimal policy, it follows that:
\begin{equation*}
    V_H(s_0,b^\prime) \geq V_H(s_0,b) \quad \forall s_0, b^\prime > b
\end{equation*}
\end{proof}

\begin{lemma}\label{lemma:3}
    For a given initial state $s_0 > d_0$ and horizon $H$, if $b > 0$, the optimal action to take at time $t=0$ is $d$.
\end{lemma}
\begin{proof}
    We will compare four cases:
\begin{enumerate}
    \item We first take an action m and then action o. Doing so yields the maximal collected reward $V_{H}^{md}$ given by 
    \begin{equation}
    V_{H}^{md}(s_0,b) = 2r + V_{H-2}(s_{max}-d_0,b-1). \label{mo}
    \end{equation}
    \item We first take an action o and then action m. Doing so yields the maximal collected reward $V_{H+1}^{dm}$ given by
    \begin{equation}
    V_{H}^{dm}(s_0,b) = 2r + V_{H-2}(s_{max},b-1).\label{om}
    \end{equation}
    \item We first take an action m and then action m. Doing so yields the maximal collected reward $V_{H}^{mm}$ given by
    \begin{equation}
    V_{H}^{mm}(s_0,b) = 2r + V_{H-2}(s_{max},b-2).\label{mm}
    \end{equation}
    \item We first take an action o and then action o. Doing so yields the maximal collected reward $V_{H+1}^{dd}$ given by
    \begin{equation}
    V_{H}^{dd}(s_0,b) = 2r + V_{H-2}(s_0-2d_0,b).\label{oo}
    \end{equation}
\end{enumerate}
Using Lemma \ref{lemma:1}, we know that the value function is an increasing function of the state for the same budget. Also from \ref{lemma:2} we know that the value function is an increasing function of the budget for the same state. Using these results and \eqref{mo}, \eqref{om} and \eqref{mm}, we get that 
\begin{equation}
    V_{H}^{dm} = \max\{V_{H}^{md}, V_{H}^{dm}, V_{H}^{mm}\}. \label{o_optimal}
\end{equation}
We don't need to compare $V_{H}^{dd}(s_0,b)$ because \eqref{o_optimal} is sufficient to prove that for $s_0 > d_0$, action $o$ is the optimal initial action.
\end{proof}

\begin{lemma}\label{lemma:4}
    For a given initial state $0 < s_0 \leq d_0$ and horizon $H$, if $b > 0$, the optimal action to take at time $t=0$ is $m$.
\end{lemma}
\begin{proof}
    We will consider two cases:
    \begin{enumerate}
        \item We take action $d$ in the first step. Doing so yields the maximal collected reward $V_H^o$ given by
        \begin{equation}
            V^d_{H}(s_0,b) = r + V_{H-1}(0,b), \label{ofirst}\\
        \end{equation}
        \item We take action $m$ in the first step. Doing so yields the maximal collected reward $V_H^m$ given by
        \begin{equation}
            V^m_{H}(s_0,b) = r + V_{H-1}(s_{max},b-1), \label{mfirst}
        \end{equation}
    \end{enumerate}
Clearly from \eqref{ofirst} and \eqref{mfirst}, taking action $d$ in the first step leads to the absorbing state and hence we get a total reward of $r$. Taking action $m$ gives the same reward $r$ and drives the state to $s_{max}$. This implies that the system will not reach the absorbing state for at least one more step and hence leads to higher total reward. Thus, we have
\begin{equation*}
    V^m_{H}(s_0,b) > V^d_{H}(s_0,b).
\end{equation*}
Hence, we can say that if $b > 0$, action $m$ is the optimal action to perform at $t=0$ when $0 < s_0 \leq d_0$.
\end{proof}

\begin{theorem} \label{theorem:1}
The value function for the MDP, $V_H(s_0,b)$, is concave in the budget $b$ for any horizon length $H \in \mathbb{N}_0$ and initial state $s_0 \in \states_{MDP}$. More specifically:
\begin{itemize}
    \item For $s_0 > d_0$, $V_H(s_0,b)$ is constant  with respect to the budget for all $b \geq \lfloor H/2 \rfloor$, where $\lfloor . \rfloor$ represents the floor function, 
    \item For $s_0 \leq d_0$, $V_H(s_0,b)$ is constant with respect to the budget for all $b \geq \lceil H/2 \rceil$, where $\lceil . \rceil$ represents the ceiling function,
    \item For budget values $b$,$b^\prime$ and $b^{\prime \prime}$ such that $b^{\prime\prime} = b^\prime + 1 = b + 2$ and $b \geq 0$, we have:
    \begin{equation}
        V_H(s_0,b^{\prime \prime}) - V_H(s_0,b^\prime) \leq V_H(s_0,b^{\prime}) - V_H(s_0,b), \label{relation}
    \end{equation}
\end{itemize}
\end{theorem}

\begin{proof}

First, we will show that the claims hold for a horizon of length 0. Then, using induction we will prove that they hold for a horizon of length $H$. 

Consider a horizon of length 0. For this case we have:
\begin{equation*}
    V_0(s_0,b) = \begin{cases}
        r,& \text{ if } s_0 > 0\\
        0,& \text{ if } s_0 = 0,
    \end{cases}
\end{equation*}
where $b \geq 0$. Clearly, $V_0(s_0,b)$ is constant in the budget for all $b \geq 0$ and thus, the first two claims hold. Also, for all values of $s_0 > 0$ we can say that:
\begin{equation*}
    V_0(s_0,b^{\prime\prime}) - V_0(s_0,b^\prime) = 0 \leq V_0(s_0,b^\prime) - V_0(s_0,b),
\end{equation*}
where $b^{\prime\prime} = b^\prime + 1 = b + 2$ and $b > 0$. 

Now, for a horizon of length $H$, assume the following:
\begin{itemize}
    \item For $s_0 > d_0$, $V_H(s_0,b)$ is constant  with respect to the budget for all $b \geq \lfloor H/2 \rfloor$, where $\lfloor . \rfloor$ represents the floor function, 
    \item For $s_0 \leq d_0$, $V_H(s_0,b)$ is constant with respect to the budget for all $b \geq \lceil H/2 \rceil$, where $\lceil . \rceil$ represents the ceiling function,
    \item Relation \eqref{relation} holds true for $b \geq 0$.
\end{itemize}
As can be clearly seen, the above assumptions hold true for $H=0$ as we proved previously. Now, assume that they hold for some $H > 0$ and consider a horizon of length $H+1$. We will prove that the assumptions hold true for horizon $H+1$ when $s_0>d_0$ and when $s_0 \leq d_0$.

First we will consider $s_0 > d_0$. Using Lemma \eqref{lemma:3} we know that, if $b > 0$, action $d$ is the optimal action to take at time $t=0$ when $s_0 > d_0$. Hence for horizon $H+1$, if $s_0 > d_0$, the value function is given by:
\begin{equation}
    V_{H+1}(s_0,b) = r + V_H(s_0-d_0,b), \quad \forall b \geq 0 \label{V:s>d}
\end{equation}
We will now consider two cases:
\begin{enumerate}
    \item Case 1 : $s_0-d_0 > d_0$. In this case, using the results for horizon $H$, we get that $V_{H+1}(s_0,b)$ becomes constant in budget for all $b \geq \lfloor H/2 \rfloor$. Using the properties of the floor function, we have:
    \begin{equation*}
        \lfloor (H+1)/2 \rfloor \geq \lfloor H/2 \rfloor.
    \end{equation*}
    \item Case 2 : $s_0-d_0 \leq d_0$. In this case, using the results for horizon $H$, we get that $V_{H+1}(s_0,b)$ becomes constant in budget for all $b \geq \lceil H/2 \rceil$. Using the properties of the ceiling and floor functions, we have:
    \begin{equation*}
    \lceil H/2 \rceil = \lfloor (H+1)/2 \rfloor.
\end{equation*}
\end{enumerate}
Thus we can say that $V_{H+1}(s_0,b)$ becomes constant with respect to the budget for all $b \geq \lfloor (H+1)/2 \rfloor$. Also, if we consider three budget values $b^{\prime\prime} = b^\prime + 1 = b + 2$ such that $b \geq 0$, we have:
\begin{equation*}
    V_{H+1}(s_0,b^{\prime\prime}) - V_{H+1}(s_0,b^\prime) = V_H(s_0-d_0,b^{\prime\prime}) - V_H(s_0-d,b^\prime)) \leq V_H(s_0-d_0,b^\prime) - V_H(s_0-d_0,b),
\end{equation*}
where the inequality is due to the assumption that relation \eqref{relation} holds holds for horizon $H$. Also, using \eqref{V:s>d}, we know that :
\begin{equation*}
    V_{H+1}(s_0,b^\prime) - V_{H+1}(s_0,b) = V_H(s_0-d_0,b^\prime) - V_H(s_0-d_0,b)
\end{equation*}
Hence, from this we can say that:
\begin{equation*}
    V_{H+1}(s_0,b^{\prime\prime}) - V_{H+1}(s_0,b^\prime) \leq V_{H+1}(s_0,b^\prime) - V_{H+1}(s_0,b),
\end{equation*}
and thus, relation \eqref{relation} holds true for horizon $H+1$ when $s_0 > d_0$.

Using Lemma \eqref{lemma:4} we know that, if $b > 0$, action $m$ is the optimal action to take at time $t=0$ when $0 < s_0 \leq d_0$. Hence for horizon $H+1$, if $0 < s_0 \leq d_0$, the value function is given by:
\begin{equation}
    V_{H+1}(s_0,b) = \begin{cases}
        r + V_H(s_0-d_0,b),& \text{if } b = 0\\
        r + V_H(s_{max},b-1),& \text{if } b > 0.
    \end{cases}, \label{V:s<d}
\end{equation}
Note that if $s_0 = 0$, then we are in the absorbing state and we get $V(0,b) = 0$ for all $b \geq 0$. Thus, we can say that $V(0,b)$ is constant with respect to budget for all $b \geq \lceil (H+1)/2 \rceil$.

 Now, for $b > 0$, using \eqref{V:s<d}  and the results for horizon $H$, we get that $V_{H+1}(s_0,b)$ becomes constant in budget for all $b-1 \geq \lfloor H/2 \rfloor$ or $b \geq \lfloor H/2 \rfloor + 1$. From the properties of the floor and ceiling functions, we know that:
 \begin{equation*}
     \lfloor H/2 \rfloor + 1 = \lceil (H+1)/2 \rceil
 \end{equation*}
 Hence, we can say that for a horizon $H+1$ and $s_0 \leq d_0$, the value function $V_{H+1}(s_0,b)$ becomes constant in b for all $b \geq \lceil (H+1)/2 \rceil$. 
 
Now, to prove that relation \eqref{relation} holds for $s_0 \leq d_0$, we will consider two cases:
\begin{enumerate}
    \item Consider three budget values $b^{\prime\prime} = b^\prime + 1 = b + 2$ such that $b>0$. Then, using \eqref{V:s<d} we have:
    \begin{equation*}
        V_{H+1}(s_0,b^{\prime\prime}) - V_{H+1}(s_0,b^\prime) = V_H(s_{max},b^{\prime\prime}-1) - V_H(s_{max},b^\prime-1)) \leq V_H(s_{max},b^\prime-1) - V_H(s_{max},b-1),
    \end{equation*}
    where the inequality is due to the assumption that relation \eqref{relation} holds for horizon $H$. Also, using \eqref{V:s<d} we know that:
    \begin{equation*}
        V_{H+1}(s_0,b^\prime) - V_{H+1}(s_0,b) = V_H(s_{max},b^\prime-1) - V_H(s_{max},b-1).
    \end{equation*}
    Hence, we using the above equations we can say that:
    \begin{equation*}
        V_{H+1}(s_0,b^{\prime\prime}) - V_{H+1}(s_0,b^\prime) \leq V_{H+1}(s_0,b^\prime) - V_{H+1}(s_0,b)
    \end{equation*}
    \item Consider three budget values $b^{\prime\prime} = b^\prime + 1 = b + 2$ such that $b=0$. If we are at $s = s_{max}$, we will reach $s=0$ in $\lceil \frac{s_{max}}{d_0} \rceil$ steps if we take only action $o$ repeatedly. We will now consider two cases:
    \begin{enumerate}
        \item Case 1 :  $\lceil \frac{s_{max}}{d_0} \rceil < H$.
        Using \eqref{V:s<d} we have:
    \begin{gather*}
        V_{H+1}(s_0,b) = r + V_H(0,b) = r\\
        V_{H+1}(s_0,b^\prime) = r + V_H(s_{max},b^\prime-1) = (\lceil \frac{s_{max}}{d_0} \rceil + 1)r\\
        V_{H+1}(s_0,b^{\prime\prime}) = r + V_H(s_{max},b^{\prime\prime}-1) = (2\lceil \frac{s_{max}}{d_0} \rceil)r.
    \end{gather*}
    Using the above equations, we have:
    \begin{gather*}
        V_{H+1}(s_0,b^{\prime\prime}) - V_{H+1}(s_0,b^\prime) = (\lceil \frac{s_{max}}{d_0} \rceil-1)r\\
        V_{H+1}(s_0,b^\prime) - V_{H+1}(s_0,b) = (\lceil \frac{s_{max}}{d_0} \rceil)r,
    \end{gather*}
    and thus we can say that:
    \begin{equation*}
        V_{H+1}(s_0,b^{\prime\prime}) - V_{H+1}(s_0,b^\prime) \leq V_{H+1}(s_0,b^\prime) - V_{H+1}(s_0,b)
    \end{equation*}

        \item Case 2 : $\lceil \frac{s_{max}}{d_0} \rceil \geq H$. Using \eqref{V:s<d} we have:
        \begin{gather*}
        V_{H+1}(s_0,b) = r + V_H(0,b) = r\\
        V_{H+1}(s_0,b^\prime) = r + V_H(s_{max},b^\prime-1) = (H + 1)r\\
        V_{H+1}(s_0,b^{\prime\prime}) = r + V_H(s_{max},b^{\prime\prime}-1) = (H + 1)r.
    \end{gather*}
    Using the above equations, we have:
    \begin{gather*}
        V_{H+1}(s_0,b^{\prime\prime}) - V_{H+1}(s_0,b^\prime) = 0\\
        V_{H+1}(s_0,b^\prime) - V_{H+1}(s_0,b) = Hr,
    \end{gather*}
    and thus we can say that:
    \begin{equation*}
        V_{H+1}(s_0,b^{\prime\prime}) - V_{H+1}(s_0,b^\prime) \leq V_{H+1}(s_0,b^\prime) - V_{H+1}(s_0,b)
    \end{equation*}
    \end{enumerate}
    
\end{enumerate}
 
 Hence, we can say that relation \eqref{relation} holds for horizon $H$ when $s_0 \leq d_0$.

 Thus, using induction, we have shown that for any horizon of length $H \geq 0$ the following holds:
 \begin{itemize}
     \item The value function becomes constant with respect to the budget for $b \geq \lfloor H/2 \rfloor$ if $s_0 > d_0$ and $b \geq \lceil H/2 \rceil$ if $s_0 < d_0$.
     \item The increase in the value function decreases with increase in the budget.
 \end{itemize}
\end{proof}

\begin{corollary}
The proof of theorem \ref{theorem:1} and lemma \ref{lemma:2} implies the concavity of the value function with the budget, i.e., if $b^\prime = \alpha b + (1-\alpha)b^{\prime\prime}$ for some $\alpha \in [0,1]$, then we have
\begin{equation}
    V_H(s_0,b^\prime) = \alpha V_H(s_0,b) + (1-\alpha)V_H(s_0,b^{\prime\prime}) \label{eq:concavity}
\end{equation}
\end{corollary}
  

This proof can be easily extended to any general $c_m \in \mathbb{R}^+$ by scaling the costs and budget with $1/c_m$. Furthermore, this proof works under heavy technical assumptions of full observability and deterministic transitions. However, we empirically observe that that same property is often true for general systems (partially observable and stochastic) and we believe the same proof approach could work, and we leave it for future work.


 We will now use this concavity property to obtain the optimal budget split among the component POMDPs of an $n$-component POMDP. Doing so would provide the approximately optimal policy for the individual component POMDPs and hence the $n$-component POMDP.

\subsection{Optimal Policy Synthesis for Multi-Component POMDP}
Consider a multi-component POMDP with $n$ components and budget $B$ as described in Section ~\ref{subsec:multicomponent}. The size of the state space is $(|\states|)^n$ where $|\states|$ is the size of the state space of each component POMDP. Also, the size of the total action space is $3^n$. Directly applying a POMDP solver to such a large state and action space may not be computationally feasible. Hence, we propose an algorithm which decouples the $n$ component POMDPs by allocating a portion of the total budget to each of them prior to the beginning of the system run. We then compute the  approximate value function for each component POMDP as a function of the budget and then using that, obtain the optimal split of the total budget. 


Given a total budget $B$, we assume that the $i^{th}$ component POMDP is alloted a budget $b_i$ from the total budget. Hence,
\begin{equation}
    b_1 + b_2 + \ldots + b_n = B \label{eq1}
\end{equation}
We now have $n$ independent POMDPs, where each POMDP has its own total budget. We formulate each of them as a b-POMDP and solve each b-POMDP using the POMCP algorithm as discussed in Section ~\ref{subsec:budgeted-POMDP}. 
Let the maximal collected reward, for component $i$, obtained using the POMCP algorithm, for a given initial state $s_0$ and horizon $H$, be denoted by $V^i_H(s_0,b_i)$. We can then find the optimal budget split among the $n$ b-POMDPs by solving a welfare maximization problem. Welfare maximization is the concept of maximizing the overall well-being or welfare of a society, and is achieved by maximizing some measure of social welfare (e.g. maximal collected reward). We thus maximize the total maximal collected reward, for all components, with respect to $b_i$ while adhering to the constraint in \eqref{eq1}, i.e.,
\begin{equation}
\begin{aligned}
&\max_{b_i} \sum_{i = 1}^{n}V^i_H(s_0,b_i)\\
&\text{s.t.} \sum_{i = 1}^{n}b_i = B. \label{maximize}
\end{aligned}
\end{equation}
Using the results we proved in Section \ref{subsec:proof}, we know that $V_H(s_0,b_i)$ is a concave function of $b_i$ in the special case mentioned in Section \ref{subsec:proof} and emperically observe it to be concave in general. The welfare maximization problem then becomes a constrained convex optimization problem. Hence, it can be solved easily and is guaranteed to have a global optimum. The solution to \eqref{maximize} provides the optimal budget allocation for each b-POMDP which in-turn gives us the optimal policy for all $n$ component POMDPs. Let $\pi_i : \states^i \rightarrow A^i$ be a policy for component $i$ with budget $b_i$ obtained using the POMCP algorithm. Then, we define the overall policy, for the $n$-component POMDP, $\pi : \states \rightarrow A$ by $$\pi = \prod_{i=1}^n \pi_i.$$
While such a policy is naturally not guaranteed to be generally optimal on the multi-component POMDP, it provably satisfies the budgetary constraints and performs well in practice. To illustrate its performance on real data, we now move to the implementation and evaluation section. 

\section{Implementation and Evaluation}
In this section we illustrate the utility of the proposed approach for multi-component decision making with budgetary-constraints. In particular, we compare the policy described above with existing approaches on a scenario of multi-component building management. Our implementation utilizes the POMDP.jl \cite{egorov2017pomdps} Julia package for efficiently solving the budgeted-POMDPs using POMCP, as well as CVXPY \cite{diamond2016cvxpy} for solving the convex optimization formulation of the budget allocation problem. The initial budget-split for solving the budget allocation problem is chosen randomly while satisfying the constraint of \eqref{eq1}.

We model the components that comprise a typical administration building with a total size of 10,000 sq. ft. The building comprises multiple components such as lighting systems, roofing components, boilers, and carpeting, where each component's cost of replacement and inspection are based on empirically derived industry averages. Each component's health is defined by the Condition Index (CI) \cite{grussing2006condition}, which takes values between 0 and 100. The condition deteriorates stochastically over time, depending on various factors, and can only be observed through explicit inspections, which incur a cost. The component fails when the CI reaches below a \textit{failure threshold}, which we assume to be 0. Components can also be replaced, restoring their CI to its full value.

The building is associated with an average maintenance budget of \$2,200,000 for a given period of interest. Using historic CI data for each component, we synthesize the transition probabilities of their corresponding Partially Observable Markov Decision Processes (POMDPs). We consider 20 sustainable components from the building, with replacement costs ranging from 0.15\% to 3\% of the total budget and inspection costs ranging from 0.01\% to 0.03\% of the total budget. We scale the total budget to 10,000 units and appropriately scale the replacement and inspection costs of all components while ensuring that they are rounded to the nearest integers. The decision-maker's objective is to maximize the  time until failure of the components by effectively allocating the budget among the components and taking replacement and inspections when needed. As in Section IV.C, we model this objective as a POMDP by assigning a reward of 1 when the CI is greater than the failure threshold and 0 otherwise, and modeling the state of 0 health and budget exhaustion as absorbing states. For our experiments, we consider simulations with a horizon of up to 100 decision steps with a 1 year step size. 

\subsection{Maintenance Policy Synthesis}
In this section we compare the maintenance and inspection policies obtained from the proposed POMDP-based model with a realistic baseline approach. In the baseline approach, a building manager typically schedules component inspection at a regular interval and the true health of the component is only obtained at these regular intervals. In the absence of an inspection, the CI of a component at a given time step is estimated to be the most probable CI state as determined by its CI transition dynamics. The baseline policy used in this section replaces the component if its estimated CI is less than a pre-determined threshold. 

We use \textit{time-to-failure} (TTF), defined as the number of simulation steps until failure, as the performance metric. We run experiments to calculate the TTF for each component by averaging the values obtained over 5 independent simulations with 100 maximum possible simulation steps. We set the maximum tree depth for POMCP rollouts to 50 and use a UCB exploration constant of 10. In the baseline policy, we inspect the CI every 5 steps and replace the component if the estimated CI is below 15. 
Figure \ref{fig:ttf-budget} shows the simulation results comparing the TTF obtained for different budget values using the baseline and the proposed approach. The proposed approach provides a clear advantage over baseline strategy over the entire range of budget values for all 20 components, irrespective of the replacement costs. 
\begin{figure}[htbp]
  \centering
  \includegraphics[width=0.8\textwidth]{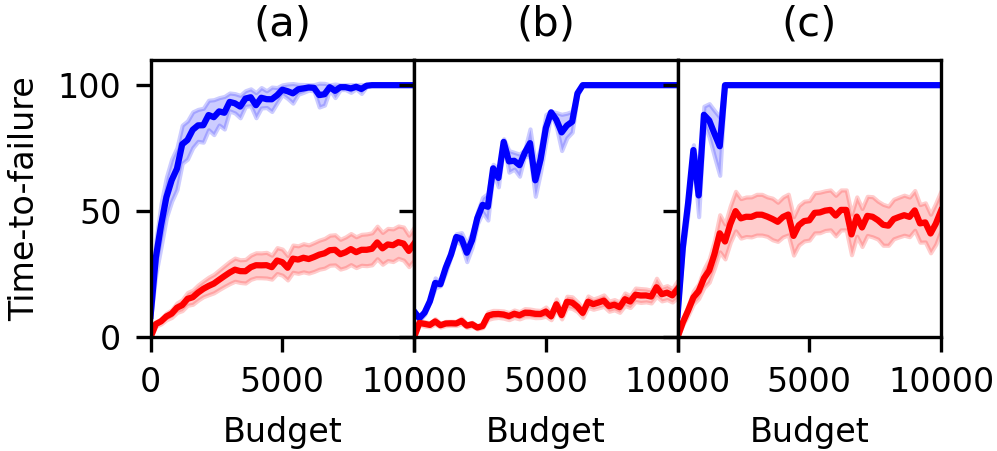}
  \caption{Comparison of the \textcolor{blue}{proposed} and \textcolor{red}{baseline} approaches using time-to-failure for a range of budget values. (a) Overall results obtained by averaging over all components. (b) Results for the Air Handling Unit component with a replacement cost of 250 units. (c) Results for the Lighting Equipment component with a replacement cost of 24 units.}
  \label{fig:ttf-budget}
  \vspace{-0.2cm}
\end{figure}

Figure \ref{fig:component_history} shows sample CI histories for the same component obtained from simulations using the proposed approach and the baseline policy. The proposed approach takes inspection and replacement actions only when deemed necessary based on the latest belief estimate and the potential loss of value due to an inaccurate estimate or due to not taking a replacement action. Such a behavior holds true for every component without any component specific parameter tuning. On the other hand, the estimated state from the baseline policy based on the most-probable transition may not always be the same as the real transition, ultimately resulting in early failures. Although it is possible to enhance the baseline by incorporating component-specific parameters and budget-aware heuristics, our experiments indicate that its performance still lags behind the proposed approach, particularly when the budget is tightly constrained.
\begin{figure}[htbp]
  \centering
  \includegraphics[width=0.8\textwidth]{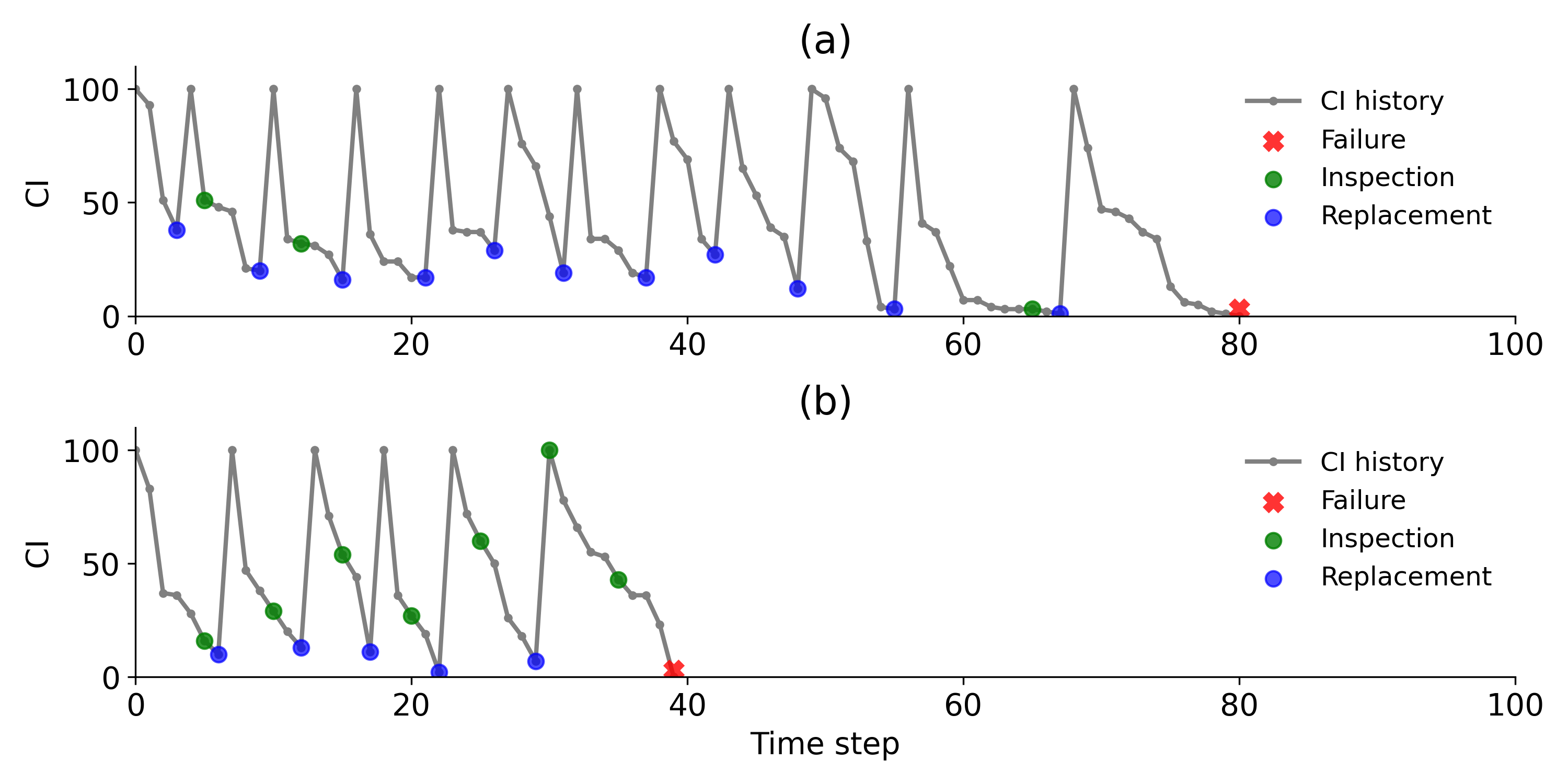}
  \caption{Sample condition index (CI) histories illustrating the performance of the proposed policy when compared to the baseline for the Boiler component with a replacement cost of 45 units, an inspection cost of 1 unit, and a total budget of 500 units. (a) CI history using proposed approach showing failure at 80 time steps. (b) Baseline approach failing at 39 time steps.}
  \label{fig:component_history}
\end{figure}
\vspace{-0.4cm}
\subsection{Budget Allocation}
We demonstrate the effectiveness of our proposed budget allocation approach by comparing it to a baseline that depends on two component properties: (i) mean-time-to-failure (MTTF) which is the expected number of steps a component takes for its condition index to go below the failure threshold when starting from maximum possible condition index, and (ii) the replacement cost of the component. The baseline allocation is proportional to ratio of the component's replacement cost and MTTF.

We quantify the performance of the budget allocation algorithms by running 20 independent simulations over all the components using the allocated budgets and calculating the overall TTF for the building. To ensure fairness, we compare both budget allocation algorithms by running simulations using policies obtained by the same decision making strategy: the POMCP-based approach. 
\begin{figure}[htbp]
  \centering
  \includegraphics[width=0.8\textwidth]{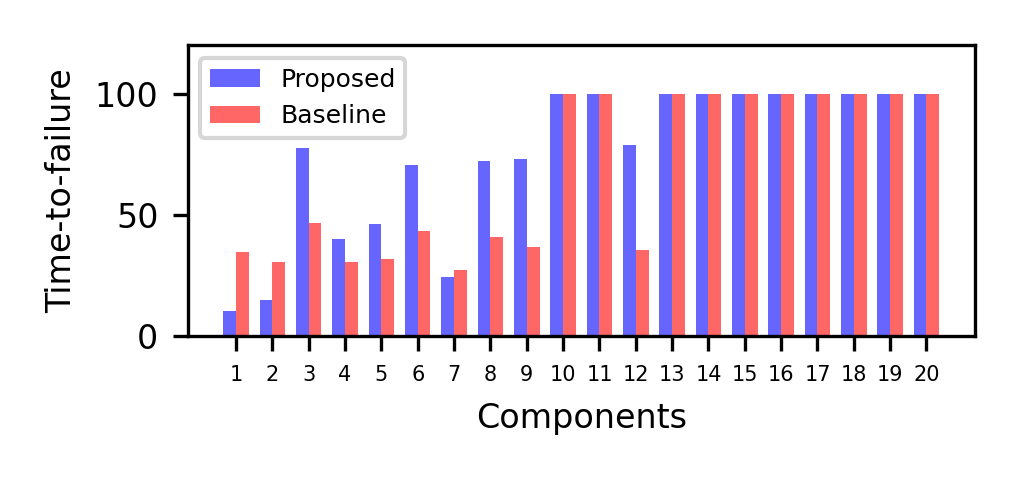}
  \vspace{-0.5cm}
  \caption{Comparison of baseline and proposed budget allocation approaches for the all 20 components for an overall budget of 10,000 units.}
  \label{fig:allocation}
  \vspace{-0.2cm}
\end{figure}
Figure \ref{fig:allocation} summarizes the TTF results from the baseline and the proposed approaches for all components in the building. The proposed allocation approach achieves an overall TTF of 1510, outperforming the baseline that achieves an overall TTF of 1355. Hence the proposed approach maximizes the overall TTF in accordance with the objective defined by \eqref{maximize}. Analyzing individual component data we observe that instances where the proposed approach underperforms the baseline exhibit only slight differences in TTF values. In contrast, when the proposed approach outperforms the baseline, we observe a significant improvement. Note that the maximal TTF of 100 is achieved by both strategies for 50\% of components, the proposed strategy performs better for 35\% of the components, and the baseline performs better only for 15\% of the components.

\begin{table}[!h]
    \centering
    \vspace{7pt}
    \begin{tabular}{c c}
        \toprule
        Number of Components & Time (mean $\pm$ std. dev. of 7 runs)\\
        \midrule
        5 & $333 ms \pm 21.4 ms$\\
        10 & $412 ms \pm 37.1 ms$ \\
       20 & $552 ms \pm 31.6 ms$ \\
        \bottomrule
    \end{tabular}
    \vspace{0.5cm}
    \caption{Comparison of time taken to find optimal budget split among 5, 10 and 20 components respectively, for a total budget of 10000 units.}
    \vspace{-0.8cm}
    \label{tab:time}
\end{table}
The results in Table \ref{tab:time} present the time taken to solve the optimization problem given by \eqref{maximize}. As can be clearly observed, the solution time increases with an increase in the number of components. However, the values are always of the order of milliseconds.

\section{Conclusions and Future Work}
In this paper, we formulated the problem of optimal policy synthesis for a multi-component POMDP with a budget, in the sense of maximizing the time before reaching an absorbing state. We first introduced a b-POMDP model to facilitate optimal planning in POMDPs while adhering to budget constraints. Next, we showed that the value function or maximal collected reward for b-POMDPs, under significant assumptions of full observability and deterministic transitions, is a concave function of the budget. We then presented an algorithm to find the optimal budget split among the component POMDPs of an $n$-component POMDP with a given total budget $B$. The budget-splitting problem was posed as a welfare maximization problem. The concavity of the maximal collected reward, with respect to the budget, makes the problem a convex optimization problem. The experimental evaluations of the proposed algorithm, in an infrastructure component management scenario, verify its effectiveness in terms of performance. Performing simulations on real data, we observe that our algorithm vastly outperforms the policies currently used in practice. 

There are two possible directions of future work. While the proposed approach makes it possible to compute an approximately optimal policy in a feasible amount of time, the computational cost of using the POMCP algorithm is still high. Our first direction of work is to reduce this cost by incorporating a learning framework so as eliminate repeated runs of POMCP. Finally, the budget allocation scheme is fixed in the sense that the budget-split is done before the start of the planning horizon. The second direction of work is to consider other efficient budget allocation methods. A sequential algorithm for optimal computing budget allocation is presented in \cite{futurework1}. Applying this algorithm to our problem may result in more accurate budget allocation. Similarly, another method which can be explored is the allocating method presented in \cite{futurework2}. Also, following these methods may allow us to generalize our algorithm for cases where the value function does not satisfy the concavity property. Furthermore, another direction is to derive an optimal dynamic budget allocation scheme to account for change in transition probabilities of the component states during the planning horizon and also account for a cyclical budget.

\section*{Acknowledgments}
This work was supported through a cooperative agreement between the University of Illinois Urbana-Champaign and the the United States Army Engineer Research and Development Center. We acknowledge the kind help of Trevor Betz, Louis Bartels, Ryan Smith, and Zachary Sunberg in the discussions leading to this paper.

\bibliographystyle{unsrt}  
\bibliography{references}

\begin{thebibliography}{10}

\bibitem{Cassandra2003ASO}
Anthony~R. Cassandra.
\newblock A survey of {POMDP} applications.
\newblock In {\em AAAI 1998 {Fall} symposium on planning with partially
  observable Markov decision processes}, 1998.

\bibitem{putterman}
Martin~L. Puterman.
\newblock {\em Markov Decision Processes: Discrete Stochastic Dynamic
  Programming}.
\newblock 1st edition, 1994.

\bibitem{astrompomdp}
Karl~J. Åström.
\newblock Optimal control of {Markov} processes with incomplete state
  information.
\newblock {\em Journal of Mathematical Analysis and Applications},
  10(1):174--205, 1965.

\bibitem{complexitypomdp}
Christos~H. Papadimitriou and John~N. Tsitsiklis.
\newblock The complexity of {Markov} {Decision} {Processes}.
\newblock {\em Mathematics of Operations Research}, 12(3):441–450, 1987.

\bibitem{finitestatecontroller}
Eric~A. Hansen.
\newblock Solving {POMDPs} by searching in policy space.
\newblock {\em arXiv}, abs/1301.7380, 1998.

\bibitem{pruning}
Anthony~R. Cassandra, Michael~L. Littman, and Nevin~Lianwen Zhang.
\newblock Incremental pruning: {A} simple, fast, exact method for partially
  observable markov decision processes.
\newblock {\em arXiv}, abs/1302.1525, 1997.

\bibitem{heuristicVI}
Trey Smith and Reid Simmons.
\newblock Heuristic search value iteration for {POMDPs}.
\newblock In {\em 20th Conference on Uncertainty in Artificial Intelligence},
  page 520–527, 2004.

\bibitem{NIPS2010_edfbe1af}
David Silver and Joel Veness.
\newblock Monte-carlo planning in large pomdps.
\newblock In {\em 23rd International Conference on Neural Information
  Processing Systems - Volume 2}, page 2164–2172, 2010.

\bibitem{mishra}
Mayank Mishra, Paulo~B. Lourenço, and G.V. Ramana.
\newblock Structural health monitoring of civil engineering structures by using
  the internet of things: A review.
\newblock {\em Journal of Building Engineering}, 48, 2022.

\bibitem{infra}
Mayara S.~Siverio Lima, Alexander Buttgereit, Cesar Queiroz, Viktors
  Haritonovs, and Florian Gschösser.
\newblock Optimizing financial allocation for maintenance and rehabilitation of
  {Munster}'s road network using the world bank's {RONET} model.
\newblock {\em Infrastructures}, 7(3), 2022.

\bibitem{hcrl}
Yifan Zhou, Bangcheng Li, and Tian~Ran Lin.
\newblock Maintenance optimisation of multicomponent systems using hierarchical
  coordinated reinforcement learning.
\newblock {\em Reliability Engineering \& System Safety}, 217, 2022.

\bibitem{deeprl}
Xiaoming Lei, Ye~Xia, Lu~Deng, and Limin Sun.
\newblock A deep reinforcement learning framework for life-cycle maintenance
  planning of regional deteriorating bridges using inspection data.
\newblock {\em Structural and Multidisciplinary Optimization}, 65(5), 2022.

\bibitem{subway}
Omar {El Hamshary}, Mona Abouhamad, and Mohamed Marzouk.
\newblock Integrated maintenance planning approach to optimize budget
  allocation for subway operating systems.
\newblock {\em Tunnelling and Underground Space Technology}, 121, 2022.

\bibitem{intinspect}
Milad Memarzadeh and Matteo Pozzi.
\newblock Integrated inspection scheduling and maintenance planning for
  infrastructure systems.
\newblock {\em Computer-Aided Civil and Infrastructure Engineering},
  31(6):403--415, 2016.

\bibitem{maintplan}
Roland Schöbi and Eleni~N. Chatzi.
\newblock Maintenance planning using continuous-state partially observable
  {Markov} decision processes and non-linear action models.
\newblock {\em Structure and Infrastructure Engineering}, 12(8):977--994, 2016.

\bibitem{ANDRIOTIS2021107551}
Charalampos~P. Andriotis and Konstantinos~G. Papakonstantinou.
\newblock Deep reinforcement learning driven inspection and maintenance
  planning under incomplete information and constraints.
\newblock {\em Reliability Engineering \& System Safety}, 212, 2021.

\bibitem{CC-POMCP}
Jongmin Lee, Geon-hyeong Kim, Pascal Poupart, and Kee-Eung Kim.
\newblock Monte-carlo tree search for constrained {POMDPs}.
\newblock In {\em 32nd International Conference on Neural Information
  Processing Systems}, volume~31. Curran Associates, Inc., 2018.

\bibitem{optmaint}
Pablo~G. Morato, Konstantinos~G. Papakonstantinou, Charalampos~P. Andriotis,
  J.S. Nielsen, and P.~Rigo.
\newblock Optimal inspection and maintenance planning for deteriorating
  structural components through {Dynamic Bayesian Networks and Markov decision
  processes}.
\newblock {\em Structural Safety}, 94:102140, 2022.

\bibitem{stochasticranking}
Yijie Peng, Edwin K.~P. Chong, Chun-Hung Chen, and Michael~C. Fu.
\newblock Ranking and selection as stochastic control.
\newblock {\em IEEE Transactions on Automatic Control}, 63(8):2359--2373, 2018.

\bibitem{cmdps}
Craig Boutilier and Tyler Lu.
\newblock Budget allocation using weakly coupled, constrained {Markov} decision
  processes.
\newblock In {\em 32nd Conference on Uncertainty in Artificial Intelligence},
  pages 52--61, 2016.

\bibitem{cmdp1}
E.~Altman.
\newblock {\em Constrained Markov Decision Processes}.
\newblock Chapman and Hall, 1999.

\bibitem{pomdpcassandra}
Anthony~R. Cassandra, Leslie~Pack Kaelbling, and Michael~L. Littman.
\newblock Acting optimally in partially observable stochastic domains.
\newblock In {\em 12th AAAI National Conference on Artificial Intelligence},
  page 1023–1028, 1994.

\bibitem{Braziunas-POMDPsurvey}
Darius Braziunas.
\newblock {POMDP} solution methods: a survey.
\newblock Technical report, University of Toronto, 2003.

\bibitem{egorov2017pomdps}
Maxim Egorov, Zachary~N. Sunberg, Edward Balaban, Tim~A. Wheeler, Jayesh~K.
  Gupta, and Mykel~J. Kochenderfer.
\newblock {POMDP}s.jl: A framework for sequential decision making under
  uncertainty.
\newblock {\em Journal of Machine Learning Research}, 18(26):1--5, 2017.

\bibitem{diamond2016cvxpy}
Steven Diamond and Stephen Boyd.
\newblock {CVXPY}: {A} {P}ython-embedded modeling language for convex
  optimization.
\newblock {\em Journal of Machine Learning Research}, 17(83):1--5, 2016.

\bibitem{grussing2006condition}
Michael~N. Grussing, Donald~R. Uzarski, and Lance~R. Marrano.
\newblock Condition and reliability prediction models using the {Weibull}
  probability distribution.
\newblock In {\em Applications of advanced technology in transportation}, pages
  19--24. 2006.

\bibitem{futurework1}
Chun-Hung Chen, Jianwu Lin, Enver Y{\"u}cesan, and Stephen~E. Chick.
\newblock Simulation budget allocation for further enhancing the efficiency of
  ordinal optimization.
\newblock {\em Discrete Event Dynamic Systems}, 10(3):251--270, 2000.

\bibitem{futurework2}
Chun-Hung Chen.
\newblock A lower bound for the correct subset-selection probability and its
  application to discrete-event system simulations.
\newblock {\em IEEE Trans. Autom. Control.}, 41:1227--1231, 1996.

\end{thebibliography}

\end{document}